\documentclass[11pt]{amsart}
\usepackage{fullpage}
\usepackage{amssymb,amsmath,amsfonts,amsthm,mathrsfs,moreverb}
\usepackage[all]{xy} \SelectTips{cm}{10}
\usepackage{color}

\numberwithin{equation}{section}
\newcommand{\CC}{\mathbb C}

\newcommand{\FF}{\mathbb F}

\newcommand{\PP}{\mathbb P}
\newcommand{\QQ}{\mathbb Q}

\newcommand{\ZZ}{\mathbb Z} 
\newcommand{\Zhat}{\widehat\ZZ}

\newcommand{\OO}{\mathcal O}

\newcommand{\calP}{\mathcal P}

\newcommand{\calC}{\mathcal C}

\newcommand{\calW}{\mathcal W}

\newcommand{\calV}{\mathcal V}

\newcommand{\ang}[1]{ \langle #1 \rangle  }

\def\cyc{{\operatorname{cyc}}}
\def\ab{{\operatorname{ab}}}
\def\un{{\operatorname{un}}}

\def\Spec{\operatorname{Spec}} 
\def\Gal{\operatorname{Gal}}
 
\def \GL {\operatorname{GL}}  

\def \GSp {\operatorname{GSp}}

\def \Sp {\operatorname{Sp}}
\def\Aut{\operatorname{Aut}} 

\def\Frob{\operatorname{Frob}}

\def\tr{\operatorname{tr}}

\def\mult{\operatorname{mult}}

\newcommand{\defi}[1]{\textsf{#1}} % for defined terms

\def\bbar#1{\setbox0=\hbox{$#1$}\dimen0=.2\ht0 \kern\dimen0 
\overline{\kern-\dimen0 #1}}
\newcommand{\Qbar}{{\overline{\mathbb Q}}} 
\newcommand{\Kbar}{{\bbar{K}}} 
 
\newcommand{\kbar}{\bbar{k}} 
\newcommand{\FFbar}{\overline{\FF}} 

\newtheorem{thm}{Theorem}[section]
\newtheorem{lemma}[thm]{Lemma}

\newtheorem{prop}[thm]{Proposition}

\theoremstyle{definition}

\theoremstyle{remark}
\newtheorem{remark}[thm]{Remark}

\newenvironment{romanenum}{\hfill \begin{enumerate} }{\end{enumerate}}

\definecolor{webcolor}{rgb}{0.8,0,0.2}
\definecolor{webbrown}{rgb}{.6,0,0}
\usepackage[
        colorlinks,
        linkcolor=webbrown,  filecolor=webcolor,  citecolor=webbrown, 
        backref,
        pdfauthor={David Zywina}, 
       pdftitle={An explicit Jacobian of dimension 3 with maximal Galois action},
]{hyperref}
\usepackage[alphabetic,backrefs,lite]{amsrefs} % for bibliography

\begin{document}
\title[An explicit Jacobian of dimension $3$ with maximal Galois action]{An explicit Jacobian of dimension $3$ with maximal Galois action}


%11G10 Abelian varieties of dimension > 1
%11F80 Galois representations

\author{David Zywina}
\address{Department of Mathematics, Cornell University, Ithaca, NY 14853}
\email{zywina@math.cornell.edu}

\begin{abstract}
We gives an explicit genus $3$ curve over $\QQ$ such that the Galois action on the torsion points of its Jacobian is a large as possible.  That such curves exist is a consequence of a theorem of D.~Zureick-Brown and the author; however, those methods do not produce explicit examples.    We shall apply the general strategies of Hall and Serre  in their open image theorems.  We also make use of Serre's conjecture to show that the modulo $\ell$ Galois actions are irreducible.   While we computationally focus on a single curve, the methods of this paper can be applied to a large family of genus $3$ curves.
\end{abstract}

\maketitle

%-----------------------------------------

\section{Introduction}

Consider a principally polarized abelian variety $A$ of dimension $g\geq 1$ defined over $\QQ$.  Fix an algebraic closure $\Qbar$ of $\QQ$ and define the absolute Galois group $G_\QQ:=\Gal(\Qbar/\QQ)$.   The Galois action on the torsion points of $A(\Qbar)$ can be expressed in terms of a Galois representation
\[
\rho_A \colon G_\QQ \to \GSp_{2g}(\Zhat),
\]
see \S\ref{SS:background} for details.   

In \cite{ZB-Z}, Zureick-Brown and the author prove that for each integer $g\geq 3$, there is a principally polarized abelian variety $A/\QQ$ (in fact the Jacobian of a trigonal curve) such that $\rho_A(G_\QQ)=\GSp_{2g}(\Zhat)$.  For such an abelian variety, the Galois group acts on the torsion points in the most general way possible.  Unfortunately, the methods of \cite{ZB-Z} are not useful for constructing examples.

The goal of this paper is to give the first \emph{explicit} $A/\QQ$ for which the representation $\rho_A$ is surjective, i.e., the Galois group acts on the torsion points of $A$ in the most general way possible.    

When $A/\QQ$ is an elliptic curve, the image of the representation $\rho_A$ is an important ingredient in several deep conjectures, for example the Lang-Trotter conjectures \cite{MR0568299} and the Koblitz conjecture \cite{MR2805578}.    Our explicit example should be useful in providing numerical evidence for related higher dimension conjectures.

\subsection{The example} \label{SS:example}

Let $C$ be the subscheme of $\PP^2_\QQ$ defined by the quartic equation
\begin{align} \label{E:main}
x^3y - x^2y^2 + x^2z^2 + xy^3 - xyz^2 - xz^3 - y^4 + y^3z - y^2z^2 - yz^3 = 0.
\end{align}
The curve $C$ is smooth and hence has genus $3$.   Let $J$ be the Jacobian of the curve $C$; it is a principally polarized abelian variety of dimension $6$ defined over $\QQ$.  The Galois action on the torsion points of $J$ is as large as possible.

\begin{thm} \label{T:main}
With $J/\QQ$ as above, we have $\rho_J(G_\QQ)=\GSp_6(\Zhat)$.
\end{thm}

\begin{remark} \label{R:easiest}
Let $A/\QQ$ be a principally polarized abelian variety of dimension $g\geq 1$. In Proposition~\ref{P:remark}, we will show that if $g\leq 2$ or if $A$ is the Jacobian of a hyperelliptic curve, then $\rho_A$ is \emph{not} surjective.  This motivates why we have first considered the Jacobian of a smooth plane quartic. 
\end{remark}

Though we focus only on a specific curve, the methods will also apply to a large class of smooth plane quartics.   Indeed, most of this paper can be viewed as describing how to make the criterion of C.~Hall in \cite{MR2820155} effective.    The largest difference from  \cite{MR2820155} is that we use Serre's conjecture to prove that the modulo $\ell$ representations are irreducible; this is motivated by the work of Dieulefait \cite{MR1969642} on abelian surfaces.

\subsection{Overview}
We now give a brief overview of the contents of this paper; none of the following will be needed later on.

For each prime $\ell$, let $J[\ell]$ be the $\ell$-torsion subgroup of $J(\Qbar)$.   The natural $G_\QQ$-action on $J[\ell]$ can be expressed by a representation
\[
\rho_{J,\ell} \colon G_\QQ \to \GSp_6(\FF_\ell),
\]
see \S\ref{SS:background} for details.   The constraint on the image of $\rho_{J,\ell}$ arises from the Weil pairing.

We will show (Proposition~\ref{P:reduce to mod ell cases}) that $\rho_J$ is surjective if and only if $\rho_{J,\ell}$ is surjective for all primes $\ell$.    So fix any odd prime $\ell$ (the prime $\ell=2$ can be dealt with separately).

We will see in \S\ref{S:good} that the curve $C$, and hence also $J$, has good reduction at all primes away from the set $S:=\{7,11,83\}$.   Therefore, $\rho_{J,\ell}$ is unramified at all primes $p \notin S\cup\{\ell\}$.  The characteristic polynomial $\det(TI - \rho_{J,\ell}(\Frob_p)) \in \FF_\ell[T]$ is the reduction modulo $\ell$ of a computable polynomial $P_p(T) \in \ZZ[T]$ that does not depend on $\ell$.

For $p\in S$ with $p\neq \ell$, we will show in \S\ref{S:bad primes} that $\rho_{J,\ell}(I_p)$ is a cyclic group of order $\ell$, where $I_p \subseteq G_\QQ$ is an inertia subgroup for $p$.   We will prove this by using the Picard-Lefschetz formula along with the fact that the only singularities for our model (\ref{E:main}) modulo $p$ are double ordinary points.   If $p \in \{7,11\}$, then $\rho_{J,\ell}(I_p)$ will be generated by a transvection (an element with determinant $1$ that fixes a codimension $1$ subspace).

In \S\ref{S:inertia at ell}, we shall give constraints on the semi-simplification of the representation $\rho_{J,\ell}|_{I_\ell}$.

In \S\ref{S:irreducible}, we will prove that the representation $\rho_{J,\ell}$ is irreducible.   The most involved case is when the composition factor of $J[\ell]$ (as an $\FF_\ell[G_\QQ]$-module) with smallest $\FF_\ell$-dimension has dimension $2$; for this, we make use of Serre's conjecture.

In \S\ref{S:primitive}, we will prove that the representation $\rho_{J,\ell}$ is primitive.   More precisely, we show that there are no non-zero subspaces $W_1,\ldots, W_r$ of $J[\ell]$ such that $J[\ell] = W_1 \oplus \cdots \oplus W_r$ and such that the $G_\QQ$-action permutes the spaces $W_1,\ldots, W_r$.

Knowing that $\rho_{J,\ell}$ is irreducible and primitive, and that $\rho_{J,\ell}(G_\QQ)$ contains a transvection, we will be able to deduce that $\rho_{J,\ell}$ is surjective.

\begin{remark}
Instead of using Serre's conjecture for irreducibility, one could use the explicit isogeny theorem  of Gaudron and R\'emond as done by Lombardo in \cite{Lombardo}.   This gives an explicit $\ell_0$ such that $\rho_{J,\ell}$ is irreducible for all $\ell\geq \ell_0$; unfortunately, $\ell_0$ will be too large to feasibly check the irreducibility for primes $\ell<\ell_0$.    We finally remark that, independently, similar ideas as in this paper have been recently used to show that $\rho_{A,\ell}$ is surjective for all $\ell >2$, where $A$ is the Jacobian of an explicit genus $3$ \emph{hyperelliptic} curve over $\QQ$, cf.~\cite{bla}.
\end{remark}
 
\subsection*{Acknowledgements}
 Thanks to Chris Hall and Ravi Ramakrishna for several helpful discussions.    The computations in this paper were performed using the \texttt{Magma} computer algebra system \cite{Magma}.
 
\section{Background and a surjectivity criterion}    

Let $C$ be a smooth projective and geometrically integral curve defined over $\QQ$ with genus $g\geq 1$.   Let $J$ be the Jacobian of the curve $C$; it is a principally polarized abelian variety of dimension $g$ defined over $\QQ$.     In later sections, we will only consider the curve $C/\QQ$ from \S\ref{SS:example}.

\subsection{Symplectic group background}

For a commutative ring $R$, let $M$ be a finitely generated free $R$-module equipped with a non-degenerate alternating bilinear form $\ang{ \; , \, }\colon M \times M \to R$.   We define $\GSp(M)$ to be the group of $A\in \Aut_R(M)$ such that for some $\mult(A) \in R^\times$, we have $\ang{Av,Aw} = \mult(A) \ang{v,w}$ for all $v,w \in M$.    The element $\mult(A)\in R^\times$ is called the \defi{multiplier} of $A$ and gives rise to a homomorphism
\[
\mult \colon \GSp(M) \to R^\times.
\]
We call $\GSp(M)$ the \defi{group of symplectic similitudes}.  

The rank of $M$ over $R$ is an even number, say $2g$.   There is an $R$-isomorphism between $M$ and $R^{2g}$ such that the pairing on $M$ agrees with the pairing  $\ang{v,w} = v^t\cdot J \cdot w$ on $R^{2g}$, where we are viewing $v$ and $w$ as column vectors and $J$ is the $2g\times 2g$ matrix $\left(\begin{smallmatrix}0 & I_g \\  -I_g & 0\end{smallmatrix}\right)$.    This gives an isomorphism between $\GSp(M)$ and $\GSp_{2g}(R):=\GSp(R^{2g})$. As before, we have a homomorphism $\mult\colon \GSp_{2g}(R)\to R^\times$ whose kernel, which we denote by $\Sp_{2g}(R)$, is called the \defi{symplectic group}.   Observe that $\GSp_{2g}(R)= \{A \in \GL_{2g}(R): A^t\cdot J \cdot A = \mult(A) J\}$ and  $\Sp_{2g}(R)= \{A \in \GL_{2g}(R): A^t\cdot J \cdot A = J\}$.   

Fix a field $k$ and an algebraic closure $\kbar$.   Take any $A\in \GSp_{2g}(k)$ and set $\gamma:= \mult(A)$.  Let $\lambda_1,\ldots, \lambda_{2g} \in \kbar$ be the roots of $P(x):= \det(xI - A) \in k[x]$.   After renumbering the $\lambda_i$, one may assume that $\lambda_{2i-1} \lambda_{2i} = \gamma$ for $1\leq i \leq g$, cf.~\cite{MR1440067}*{Lemma~3.3}.    From this, one can verify that
\begin{equation} \label{E:pre-functional}
x^{2g} P(\gamma/x) = \gamma^g P(x).
\end{equation}

\subsection{Galois representations}  \label{SS:background}

For each integer $n\geq 1$, let $J[n]$ be the $n$-torsion subgroup of $J(\Qbar)$; it is a $\ZZ/n\ZZ$-module of rank $2g$.   There is a natural action of the Galois group $G_\QQ$ on $J[n]$ that respects its group structure.   The Weil pairing and the principal polarization of $J$ give a non-degenerate and alternating pairing  
\[
e_n \colon J[n] \times J[n] \to \mu_n,
\]
where $\mu_n$ is the group of $n$-th roots of unity in $\Qbar$.

Let $\chi_n \colon G_\QQ \to (\ZZ/n\ZZ)^\times$ be the modulo $n$ cyclotomic character, i.e., $\sigma(\zeta)=\zeta^{\chi_n(\sigma)}$ for all $\sigma\in G_\QQ$ and $\zeta\in \mu_n$.  The pairing $e_n$ satisfies 
\[
e_n(\sigma(v),\sigma(w))=\sigma(e_n(v,w)) = e_n(v,w)^{\chi_n(\sigma)}
\]
for all $v,w\in J[n]$ and $\sigma\in G_\QQ$.  The Galois action on $J[n]$ can thus be expressed by a Galois representation
\[
\rho_{J,n} \colon G_\QQ \to \GSp(J[n],e_n)\cong \GSp_{2g}(\ZZ/n\ZZ).
\]
Note that $\mult \circ \rho_{J,n} = \chi_n$.  By combining over all $n$ and choosing bases compatibly, we obtain a single Galois representation
\[
\rho_J \colon G_\QQ \to \GSp_{2g}(\Zhat).
\]
The character $\mult\circ \rho_A \colon G_\QQ \to \Zhat^\times$ is the cyclotomic character and is thus surjective.

The following proposition, which will be proved in \S\ref{SS:proof of reduce to mod ell case}, will let us restrict our attention to the representations $\rho_{J,\ell}$.

\begin{prop} \label{P:reduce to mod ell cases}
Let $C/\QQ$ be a smooth projective and geometrically integral curve of genus $g\geq 3$ and let $J$ be its Jacobian.  Then $\rho_J(G_\QQ)=\GSp_{2g}(\Zhat)$ if and only if $\rho_{J,\ell}(G_\QQ) \supseteq \Sp_{2g}(\FF_\ell)$ for all primes $\ell$.
\end{prop}

With the above proposition in mind, we now give a criteria for showing that a subgroup of $\GSp_{2g}(\FF_\ell)$ contains $\Sp_{2g}(\FF_\ell)$.  First we need to introduces a few definitions.

Fix a representation $G \to \Aut_{\FF_\ell}(V)$, where $V$ is a finite dimensional $\FF_\ell$-vector space.  We say that $V$ is \defi{reducible} (and \defi{irreducible} otherwise) if there is a non-trivial proper subspace of $V$ that is stable under the $G$-action.   We say that $V$ is \defi{imprimitive} (and \defi{primitive} otherwise) if there is an integer $r\geq 2$ and non-zero subspaces $W_1,\ldots, W_r$ of $V$ such that $V=W_1\oplus\cdots \oplus W_r$ and such that $\{\sigma(W_1),\ldots,\sigma(W_r)\}= \{W_1,\ldots, W_r\}$ for all $\sigma\in G$.

For $A\in \Aut_{\FF_\ell}(V)$, let $V^{A=1}$ be the subspace of $V$ consisting of the vectors that are fixed by $A$.   We say that $A$ is a \defi{transvection} if $V^{A=1}$ has codimension $1$ in $V$ and $\det(A)=1$.

\begin{prop} \label{P:mod ell group theory}
Fix an integer $g\geq 2$ and an odd prime $\ell$.   Let $G$ be a subgroup of $\GSp_{2g}(\FF_\ell)$ with its natural action on $V=\FF_\ell^{2g}$.  Suppose that $G$ contains a transvection and that the action of $G$ on $V$ is irreducible and primitive.   Then $G \supseteq  \Sp_{2g}(\FF_\ell)$.
\end{prop}
\begin{proof}
Let $R$ be the subgroup of $G$ generated by transvections; it is a subgroup of $G\cap \Sp_{2g}(\FF_\ell)$.  We have $R\neq 1$ since $G$ contains a transvection by assumption.   The group $R$ is normal in $G$ since the conjugate of a transvection is also a transvection.

Fix an irreducible $R$-submodule $W$ of $V$.   Using that $R$ is normal in $G$, one can verify that $\sigma(W)$ is an $R$-module for all $\sigma\in G$.   Let $H$ be the group consisting of $\sigma\in G$ for which $\sigma(W)=W$.    Using that $V$ is an irreducible $G$-module, we find that that $V= \sum_{\sigma\in G/H} \sigma(W)$.   Lemma~6 of \cite{MR2820155}, which uses parts of \cite{MR2372151}, says that we in fact have a direct sum $V= \oplus_{\sigma\in G/H}\, \sigma(W)$.   

Therefore, $V$ is the direct sum of the subspaces $\{\sigma(W): \sigma \in G/H\}$ which are permuted by the natural action of $G$.   Since $G$ acts primitively on $V$ by assumption, we deduce that $W=V$, i.e., $V$ is an irreducible $R$-module.

The main theorem of Zalesski{\u\i} and Sere{\v{z}}kin in \cite{MR0412295} shows that $\Sp_{2g}(\FF_\ell)$ contains no proper subgroups that act irreducibly on $V$ and are generated by transvections.    Therefore, $R=\Sp_{2g}(\FF_\ell)$.  The lemma follows since $R \subseteq G$.
\end{proof}

\subsection{Compatibility}   \label{SS:compatibility}

Take any prime $p$ for which $C/\QQ$, and hence also $J/\QQ$, has good reduction.      Let $C_p$ and $J_p$ be the reduction of $C$ and $J$, respectively, modulo $p$.  The abelian variety $J_p/\FF_p$ agrees with the Jacobian of $C_p/\FF_p$.  

Take any prime $\ell \neq p$.   Let 
\[
\rho_{J,\ell^\infty} \colon  G_\QQ \to \GSp_{2g}(\ZZ_\ell)
\]
be the Galois representation obtained by composing $\rho_J$ with the natural projection $\GSp_{2g}(\Zhat)\to \GSp_{2g}(\ZZ_\ell)$; it can also be obtained by taking the inverse limit of the representations $\rho_{J,\ell^n}$.         The representation $\rho_{J,\ell^\infty}$ is unramified at $p$ and we have
\[
\det(T I - \rho_{J,\ell^\infty}(\Frob_p)) = P_{J_p}(T),
\]
for some polynomial $P_{J_p}(T) \in \ZZ[T]$ that does not depend on the choice of $\ell$.    Here $\Frob_p$ is an \defi{(arithmetic) Frobenius automorphism} of $p$.

Let $\pi_p\colon J_p\to J_p$ be the Frobenius endomorphism of $J_p/\FF_p$.   We may also characterize $P_{J_p}(T)$ as the polynomial in $\QQ[T]$ for which $P_{J_p}(n)$ is the degree of the isogeny $n-\pi_p$ for every integer $n$.  

We can also describe the polynomial $P_{J_p}(T)$ in terms of the zeta function of $C_p$.   Recall that the \defi{zeta function} of $C_p/\FF_p$ is the formal power series 
\[
Z_{C_p}(T)=\exp\Big(\sum_{m=1}^\infty |C_p(\FF_{p^m})| \cdot T^m/m\Big).
\]
From Weil, we know that  $Z_{C_p}(T) = P_{J_p}^{\operatorname{rev}}(T)/\big( (1-T)(1-pT)\big)$, where $P_{J_p}^{\operatorname{rev}}(T):= T^{2g} P_{J_p}(1/T)$.

We have $\mult \circ \rho_{J,\ell}(\Frob_p) = p$, so from (\ref{E:pre-functional}) we obtain the functional equation
\begin{equation} \label{E:functional}
T^{2g} P(p/T) = p^g P(T).
\end{equation}

\subsection{Proof of Proposition~\ref{P:reduce to mod ell cases}}  \label{SS:proof of reduce to mod ell case}

We first prove two group theoretic lemmas.

\begin{lemma} \label{L:perfect}
Take any integers $g\geq 3$ and $n\geq 2$.   
\begin{romanenum}
\item \label{L:perfect i}
The group $\Sp_{2g}(\ZZ/n\ZZ)$ is perfect.   
\item \label{L:perfect ii}
The only simple groups that are quotients of $\Sp_{2g}(\ZZ/n\ZZ)$ are the groups $\Sp_{2g}(\ZZ/\ell\ZZ)/\{\pm I\}$ with $\ell | N$.
\end{romanenum}
\end{lemma}
\begin{proof}
Fix $g\geq 3$.    Part~(\ref{L:perfect i}) follows by observing that $\Sp_{2g}(\ZZ)$ is its own commutator subgroup and that the reduction modulo $n$ map $\Sp_{2g}(\ZZ)\to\Sp_{2g}(\ZZ/n\ZZ)$ is surjective for all $n\geq 2$, cf.~\cite{MR0244257}*{\S12}.  Here we need $g\geq 3$ since $\Sp_{2g}(\ZZ)$ is not its own commutator subgroup when $g$ is $1$ or $2$.

We now prove (\ref{L:perfect ii}).  Using part (\ref{L:perfect i}), it suffices to show that $\Sp_{2g}(\ZZ/\ell\ZZ)/\{\pm I\}$ is the only non-abelian simple group occurring as a Jordan-H\"older factor of $\Sp_{2g}(\ZZ/\ell^e\ZZ)$ for a fixed prime $\ell$ and integer $e\geq 1$.  Note that the kernel of $\Sp_{2g}(\ZZ/\ell^e\ZZ) \to \Sp_{2g}(\ZZ/\ell\ZZ)$ is an $\ell$-group and hence solvable, so one may assume that $e=1$.  The group $\{\pm I\}$ is abelian and $\Sp_{2g}(\ZZ/\ell\ZZ)/\{\pm I\}$ is simple and non-abelian.  Here we need $g\geq 3$ again, since $\Sp_{2g}(\ZZ/\ell\ZZ)$ is solvable if $(g,\ell)\in \{(1,2),(1,3), (2,2)\}$).
\end{proof}

\begin{lemma} \label{L:Vasiu}
Fix an integer $g\geq 3$ and let $H$ be a closed subgroup of $\Sp_{2g}(\Zhat)$.   Suppose that the reduction modulo $\ell$ map $H\to \Sp_{2g}(\FF_\ell)$ is surjective for all primes $\ell$.    Then $H=\Sp_{2g}(\Zhat)$.
\end{lemma}
\begin{proof}
For each integer $n\geq 2$, let $H(n) \subseteq \Sp_{2g}(\ZZ/n\ZZ)$ be the image of $H$ under the reduction modulo $n$ map.  We claim that $H(n)=\Sp_{2g}(\ZZ/n\ZZ)$ for all $n\geq 2$.  The lemma will follow directly from the claim since $H$ is closed.

First suppose that $n$ is a prime power, say $n=\ell^e$ for some prime $\ell$ and integer $e\geq 1$.   One can show that there are no proper subgroups of $\Sp_{2g}(\ZZ/\ell^e\ZZ)$ whose image modulo $\ell$ is the full group	$\Sp_{2g}(\ZZ/\ell\ZZ)$ (this for example follows from \cite{MR2081941}*{Theorem~2.2.5} with $G=\Sp_{2g}$ and $k=\FF_\ell$).    Since $H(\ell)=\Sp_{2g}(\ZZ/\ell\ZZ)$ by hypothesis, we deduce that $H(\ell^e)=\Sp_{2g}(\ZZ/\ell^e \ZZ)$.   So the claim holds when $n$ is a prime power.

Now suppose that $n \geq 2$ is not a prime power.   By induction, we may assume that $n=m_1 m_2$ with $m_1,m_2 \geq 2$ relatively prime such that $H(m_1) = \Sp_{2g}(\ZZ/m_1\ZZ)$ and $H(m_2) = \Sp_{2g}(\ZZ/m_2\ZZ)$.   We can thus view $H(m)$ as a subgroup of $\Sp_{2g}(\ZZ/m_1\ZZ)\times \Sp_{2g}(\ZZ/m_1\ZZ)$ that projects surjectively on each of the two factors.  If $H(m)\neq \Sp_{2g}(\ZZ/m \ZZ)$, then Goursat's lemma (cf.~\cite{MR0419358}*{Lemma~3.2}) implies that $\Sp_{2g}(\ZZ/m_1\ZZ)$ and $\Sp_{2g}(\ZZ/m_2\ZZ)$ have a common simple group as a quotient;  this is impossible by Lemma~\ref{L:perfect}(\ref{L:perfect ii}).   Therefore, $H(n)=\Sp_{2g}(\ZZ/n\ZZ)$.  This completes our proof of the claim.
\end{proof}

We now prove Proposition~\ref{P:reduce to mod ell cases}.  First suppose that $\rho_{J,\ell}(G_\QQ) \supseteq \Sp_{2g}(\FF_\ell)$ for all primes $\ell$.  Let $H$ be the the commutator subgroup of $\rho_{J}(G_\QQ)$, i.e., the maximal closed normal subgroup of $\rho_{J}(G_\QQ)$ with abelian quotient.   Observe that $H$ is a closed subgroup of $\Sp_{2g}(\Zhat)$.    

Take any prime $\ell$.  The image $H(\ell)\subseteq \Sp_{2g}(\FF_\ell)$ of $H$ under the reduction modulo $\ell$ map is equal to the commutator subgroup of $\rho_{J,\ell}(G_\QQ)$.   We thus have $H(\ell) = \Sp_{2g}(\FF_\ell)$ since $\rho_{J,\ell}(G_\QQ) \supseteq \Sp_{2g}(\FF_\ell)$ by assumption and since $\Sp_{2g}(\FF_\ell)$ is perfect by Lemma~\ref{L:perfect}.   Lemma~\ref{L:Vasiu} now implies that $H=\Sp_{2g}(\Zhat)$.

Since $\rho_{J,\ell}(G_\QQ)$ contains $H=\Sp_{2g}(\Zhat)$ and $\mult(\rho_J(G_\QQ))=\Zhat^\times$ (see \S\ref{SS:background}), we deduce that $\rho_{J,\ell}(G_\QQ) = \GSp_{2g}(\Zhat)$ as desired.     Finally, the other direction of Proposition~\ref{P:reduce to mod ell cases} is easy.

\subsection{Further remarks}

We now explain the claims from Remark~\ref{R:easiest}; we will not use this later on.

\begin{prop} \label{P:remark} 
Let $A/\QQ$ be a principally polarized abelian variety of dimension $g \geq 1$.     Suppose that $g \in \{1,2\}$ or that $A$ is the Jacobian of a hyperelliptic curve.    Then $\rho_A(G_\QQ)\neq \GSp_{2g}(\Zhat)$.
\end{prop}
\begin{proof}
Suppose that $\rho_A(G_\QQ)=\GSp_{2g}(\Zhat)$.   Since $\mult\circ \rho_A \colon G_\QQ \to \Zhat^\times$ is the cyclotomic character, we have $\rho_A(G_{\QQ^\cyc}) = \rho_A(G_\QQ)\cap \Sp_{2g}(\Zhat)=\Sp_{2g}(\Zhat)$, where $\QQ^\cyc$ is the cyclotomic extension of $\QQ$.   The group $\rho_A(G_{\QQ^\ab})$ is the commutator subgroup of $\GSp_{2g}(\Zhat)$, where $\QQ^\ab$ is the maximal abelian extension of $\QQ$.     By the Kronecker-Weber theorem, we have $\QQ^\ab=\QQ^\cyc$ and hence the commutator subgroup of $\GSp_{2g}(\Zhat)$ is equal to $\Sp_{2g}(\Zhat)$.   In particular, the commutator subgroup of $\GSp_{2g}(\ZZ/n\ZZ)$ is $\Sp_{2g}(\ZZ/n\ZZ)$ for all $n\geq 2$.   However, when $g\in \{1,2\}$, the group $\GSp_{2g}(\ZZ/2\ZZ)=\Sp_{2g}(\ZZ/2\ZZ)$ is solvable and hence its commutator is a proper subgroup of $\Sp_{2g}(\ZZ/2\ZZ)$.   

Now suppose that $A$ is the Jacobian of a hyperelliptic curve $X/\QQ$ (of genus $g\geq 3$).  We have $\rho_A(G_\QQ)=\GSp_{2g}(\Zhat)$ and hence $\rho_{A,2}(G_\QQ)=\GSp_{2g}(\FF_2)=\Sp_{2g}(\FF_2)$.   Let $P_1,\ldots, P_{2g+2} \in X(\Qbar)$ be the Weierstrass points of $X$; they are the points fixed by the hyperelliptic involution.  Let $K$ be the smallest extension of $\QQ$ for which all the points $P_1,\ldots, P_{2g+2}$ lie in $X(K)$.   The extension $K/\QQ$ is Galois and the group $\Gal(K/\QQ)$ is isomorphic to a subgroup of $\mathfrak{S}_{2g+2}$.   One can show that the $2$-torsion subgroup of $A(\Qbar)$ is generated by the points represented by the divisors $P_i - P_j$.    Therefore, $\rho_{A,2}(\Gal(\Qbar/K))=\{I\}$ and hence
\[
{\prod}_{i=1}^g \big(2^{2i-1}(2^{2i}-1)\big) = |\Sp_{2g}(\FF_2)| = |\rho_{A,2}(G_\QQ)| \leq [K:\QQ] \leq (2g+2)!.
\]
Proceeding by induction on $g$, one can check that this inequality fails for all $g\geq 3$.   Therefore, $\rho_A$ is not surjective.
\end{proof}

The case $g=1$ of Proposition~\ref{P:remark}  was first observed by Serre \cite{MR0387283}*{Prop.~22}.

Proposition~\ref{P:remark} need not hold over a general number field when $g\in \{1,2\}$.  For example, Grecius \cite{MR2778661} found an explicit elliptic curve $E/k$, with $k$ a cubic extension of $\QQ$, such that $\rho_{E}(\Gal(\kbar/k))=\GL_2(\Zhat)$.

\section{Good primes} \label{S:good}
Define the set of primes
\[
S:= \{7,11,83\};
\]
these are the primes for which $C/\QQ$, and hence $J$, has bad reduction, cf.~Lemma~\ref{L:first computation}.

\subsection{Singularities}  Let $\calC$ be the subscheme of $\PP^2_\ZZ$ defined by the equation (\ref{E:main}).    For each ring $R$, let $\calC_R$ be the scheme over $\Spec R$ obtained by base extending $\calC$ to $R$.    The curve $C/\QQ$ is of course $\calC_\QQ$.  Let $\QQ_p^\un$ be the maximal unramified extension of $\QQ_p$ in $\Qbar_p$ and let $\ZZ_p^\un$ be its local ring.   The residue field of $\ZZ_p^\un$ is an algebraic closure $\FFbar_p$ of $\FF_p$.

\begin{lemma} \label{L:first computation}
\begin{romanenum}
\item \label{L:first computation i}
For any prime $p\notin S$, the curve $C/\QQ$ has good reduction at $p$.  Moreover, $\calC_{\ZZ_p} \to \Spec \ZZ_p$ is smooth and proper.
\item \label{L:first computation ii}
Take any prime $p \in S$.   The morphism 
\[
\calC_{\ZZ_p^\un} \to \Spec \ZZ_p^\un
\]
is smooth away from a finite set $\Sigma$ of points that lie in the special fiber $\calC_{\FFbar_p}$.   The points in $\Sigma$ are all ordinary double points of $\calC_{\FFbar_p}$.    We have $|\Sigma|=1$ if $p\in\{7,11\}$ and $|\Sigma|=2$ otherwise.

For each $P \in \Sigma$, the completion of the local ring of $\calC_{\ZZ_p^\un}$ at $P$ is isomorphic as a $\ZZ_p^\un$-algebra to $\ZZ_p^\un[[x,y]]/(xy + p)$.
\end{romanenum}

\end{lemma}
\begin{proof}
Let $f(x,y,z)$ be the polynomial on the left hand side of (\ref{E:main}).   Define the polynomial  $g(x,y,z):=f(x-69y-1389z,y-64z,z)$.   Since $f(x,y,z)=g(x+69y+ 5805z, y+ 64, z)$, there is no harm in assuming instead that $\calC$ is the subscheme of $\PP^2_\ZZ$ defined by $g(x,y,z)=0$.

 Let $I$ be the the ideal of $\ZZ[x,y,z]$ generated by $g$ and the partial derivatives $g_x$, $g_y$ and $g_z$.  Let $I'$ be the saturation of $I$ (with respect to the irrelevant ideal of $\ZZ[x,y,z]$).    One can show, as in the following \texttt{Magma} code, that 
\[
I'= \big(6391, \, x,\,83y, \,y^2 + 11yz + 616z^2\big).
\]
{\small
\begin{verbatimtab}
	R<x,y,z>:=PolynomialRing(Integers(),3);
	f:=x^3*y-x^2*y^2+x^2*z^2+x*y^3-x*y*z^2-x*z^3-y^4+y^3*z-y^2*z^2-y*z^3;
	g:=Evaluate(f,[x-69*y-1389*z,y-64*z,z]);
	I:=ideal<R|[g,Derivative(g,x),Derivative(g,y),Derivative(g,z)]>; 
	Saturation(I) eq ideal<R|[6391,x,83*y,y^2+11*y*z+616*z^2]>;
\end{verbatimtab}
}
That the primes divisors of $6391$ are the elements of $S$ is enough to show that $\calC$ is smooth away from the fibers over the primes $p\in S$.  Part (\ref{L:first computation i}) is an immediate consequence.

From our description of $I'$, we find that the only singular points of $\calC$ are:
\begin{itemize}
\item 
the point $(0:0:1)$ in the fiber $\calC_{\FF_7}$,
\item
the point $(0:0:1)$ in the fiber $\calC_{\FF_{11}}$,
\item
the points $(0:32:1)$ and $(0:40:1)$ in the fiber $\calC_{\FF_{83}}$.
\end{itemize}

Take any singular point $P$ of $\calC$.   Take a prime $p\in S$ and integer $y_0\in \{0,32,40\}$ such that $P$ is the image of $(0:y_0:1)$ modulo $p$.   Define the polynomial
\[
F(x,y):= g(x,y+y_0,1) \in \ZZ[x,y].
\]
The completion of the local ring of $\calC_{\ZZ_p^\un}$ at $P$ is thus isomorphic to $\ZZ_p^\un[[x,y]]/(F(x,y))$.

It is straightforward to check that in $\ZZ[x,y]$, we have
\[
F(x,y) \equiv  a + a_1 x +a_2 y + Q(x,y) \pmod{(x,y)^3}
\]
where is $Q(x,y) \in \ZZ[x,y]$ is a quadratic form whose image in $\FF_p[x,y]$ is also non-degenerate, and $a\equiv a_1\equiv  a_2 \equiv 0 \pmod{p}$ with $a\not\equiv 0 \pmod{p^2}$.

Proposition~2.4 of \cite{MR926276}*{III} shows that we have an isomorphism of $\ZZ_p$-algebras
\[
\ZZ_p[[x,y]]/(F(x,y)) \cong \ZZ_p[[x,y]]/(Q'(x,y)+b),
\]
where $Q'(x,y)\in \ZZ_p[x,y]$ is a quadratic form whose image in $\FF_p[x,y]$ is non-degenerate and $b\in \ZZ_p$.   More precisely, the proof of Proposition~2.4 of \cite{MR926276}*{III} shows that there are $\alpha_1,\alpha_2 \in p\ZZ_p$ and $h_1,h_2 \in (x,y)^3 \subseteq \ZZ_p[[x,y]]$ such that $F(x+\alpha_1+h_1,y+\alpha_2+h_2)$ is of the desired form $Q'(x,y)+b$;  hence $b$ has $p$-adic valuation $1$ since $a$ has $p$-adic valuation $1$ and $a_1,a_2,\alpha_2,\alpha_2 \in p\ZZ_p$.

Since $\ZZ_p^\un$ is strictly Henselian, there is a matrix $A\in \GL_2(\ZZ_p^\un)$ such that $Q'(A_{1,1} x + A_{1,2} y, A_{2,1}x + A_{2,2} y) = xy$, cf.~Proposition~2.2 of \cite{MR926276}*{III}.   We deduce that the completion of the local ring of $\calC_{\ZZ_p^\un}$ is isomorphic as a $\ZZ_p^\un$-algebra to $\ZZ_p^\un[[x,y]]/(xy+b)$.  After replacing $x$ by itself times an appropriate unit of $\ZZ_p^\un$, we may further assume that $b=p$.
\end{proof}

\subsection{Frobenius polynomials}

Now take any prime $p\notin S$.  Let $C_p$ be the curve in $\PP^2_{\FF_p}$ defined by (\ref{E:main}).   By Lemma~\ref{L:first computation}, we find that $C_p/\FF_p$ is a smooth projective curve of genus $3$.   The abelian variety $J$ thus has good reduction modulo $p$ and its reduction $J_p/\FF_p$ is equal to the Jacobian of $C_p$.\\

We take $P_p(T)$ to be the polynomial $P_{J_p}(T)$ from \S\ref{SS:compatibility}; it is monic with integer coefficients and has degree $6$. From \S\ref{SS:compatibility}, we find that for each prime $\ell \neq p$, we have
\[
\det(TI - \rho_{J,\ell}(\Frob_p)) \equiv P_p(T) \pmod{\ell}.
\]
 Using (\ref{E:functional}), we find that
\[
P_p(T) = T^6 + a_p T^5 + b_p T^4 +c_p T^3 + p b_p T^2 +p^2 a_p T + p^3
\]
for unique integers $a_p$, $b_p$ and $c_p$.

We have computed $P_p(T)$ for a few small primes $p\notin S$.
{\small
\begin{align*}
P_{ 2 }(T)&= T^6 + 3 T^5 + 6 T^4 + 9 T^3 + 12 T^2 + 12 T + 8\\
P_{ 3 }(T)&= T^6 + T^5 + 2 T^4 + 3 T^3 + 6 T^2 + 9 T + 27\\
P_{ 5 }(T)&= T^6 + 4 T^5 + 10 T^4 + 17 T^3 + 50 T^2 + 100 T + 125\\
P_{ 17 }(T)&= T^6 + 2 T^5 + 9 T^4 + 120 T^3 + 153 T^2 + 578 T + 4913\\
P_{ 19 }(T)&= T^6 + 4 T^5 + 18 T^4 + 91 T^3 + 342 T^2 + 1444 T + 6859\\
P_{ 23 }(T)&= T^6 + 5 T^5 + 19 T^4 + 53 T^3 + 437 T^2 + 2645 T + 12167\\
P_{ 41 }(T)&= T^6 + 42 T^4 - 212 T^3 + 1722 T^2 + 68921\\
P_{ 43 }(T)&= T^6 + 3 T^5 - T^4 - 43 T^3 - 43 T^2 + 5547 T + 79507\\
P_{ 73 }(T)&= T^6 - 4 T^5 - 43 T^4 + 581 T^3 - 3139 T^2 - 21316 T + 389017
\end{align*}
}
One way to compute $P_p(T)$ is by using the zeta function interpretation in \S\ref{SS:compatibility}.  After computing $|C_p(\FF)|$, $|C_p(\FF_{p^2})|$ and $|C_p(\FF_{p^3})|$, one can determine $a_p$, $b_p$ and $c_p$ (and hence $P_p(T)$) from the congruence
\[
1+a_p T + b_pT^2 + c_p T^3 \equiv (1-T)(1-pT) \exp\Big(\sum_{m=1}^3 |C_p(\FF_{p^m})|\cdot T^m/m \Big)  \mod{T^3}.
\]

The above explicit polynomials $P_p(T)$ have been computed using the follow \texttt{Magma} code:
{\small
\begin{verbatimtab}
	Pol<T>:=PolynomialRing(Rationals());
	for p in [2,3,5,17,19,23,41,43,73] do
	    P2<x,y,z>:=ProjectiveSpace(GF(p),2);	
            f:=x^3*y-x^2*y^2+x^2*z^2+x*y^3-x*y*z^2-x*z^3-y^4+y^3*z-y^2*z^2-y*z^3;
	    Cp:=Curve(P2,f); 
	    P:=Pol!LPolynomial(Cp); 
	    print T^6*Evaluate(P,1/T);
	end for;
\end{verbatimtab}
}
\subsection{Maximal image modulo $2$}

We now show that $\rho_{J,2}$ is surjective.

\begin{lemma} \label{L:mod 2 case}
We have $\rho_{J,2}(G_\QQ)=\GSp_6(\FF_2)$.
\end{lemma}
\begin{proof} 
Define $G:=\rho_{J,2}(G_\QQ)$; it is a subgroup of $\GSp_6(\FF_2)=\Sp_6(\FF_2)$.     Consider an odd prime $p \notin S$ and let $f_p(x)\in \FF_2[x]$ be the reduction of $P_p(x)$ modulo $2$.    Assume that $f_p(x)$ is separable and hence it is also the minimal polynomial of $g_p:=\rho_{J,2}(\Frob_p)$.   Therefore, the order of $g_p$ is the smallest integer $n_p \geq 1$ for which $f_p(x)$ divides $x^n-1 \in \FF_2[x]$.   From \S\ref{S:good}, we find that $f_{23}(x) = x^6 + x^5 + x^4 + x^3 + x^2 + x + 1$ and $f_{73}(x) = (x^2 + x + 1)(x^4 + x^3 + x^2 + x + 1)$.   One can then check that $n_{23}=7$ and $n_{73}=15$, and thus $G$ contains elements of order $7$ and $15$.

A computation shows that $\Sp_6(\FF_2)$ has no maximal subgroups with elements of order $ 7$ and $15$; therefore, $G=\Sp_6(\FF_2)$.  Moreover, any maximal subgroup of $G$ that has order divisible by $7\cdot 15$ is isomorphic to $\mathfrak{S}_8$ and hence has no element of order $15$ (this can be easily deduced from the description of maximal subgroups of $\Sp_6(\FF_2)$ in \cite{Atlas}*{p.46}).  
\end{proof}

\section{Inertia at bad primes} \label{S:bad primes}
For a prime $p\in S$, let $I_p$ be an inertia subgroup of $G_\QQ$ at the prime $p$.  The goal of this section is to prove the following using the Picard-Lefschetz formula.

\begin{prop} \label{P:bad primes}
For primes $p \in S$ and $\ell \neq p$, the group $\rho_{J,\ell}(I_p)$ is cyclic of order $\ell$.  If $p\in\{7,11\}$, then $\rho_{J,\ell}(I_p)$ is generated by a transvection.
\end{prop}

Fix a prime $p \in S$.   Set $R=\ZZ_p^\un$ and let $K=\QQ_p^\un$ be its quotient field.   Fix an algebraic closure $\Kbar$ of $K$.   With a choice of embedding $\Qbar\hookrightarrow \Kbar$, the restriction map gives an injective homomorphism $G_K:=\Gal(\Kbar/K) \hookrightarrow G_\QQ$ that we can view as an inclusion.   The group $G_K$ is then conjugate to $I_p$ in $G_\QQ$.      

Fix a prime $\ell\neq p$.  So to prove Proposition~\ref{P:bad primes}, we need only consider the action of $G_K$ on $J[\ell]\subseteq J(\Kbar)$.  Define the $\FF_\ell$-vector space $V:=H^1(C_{\Kbar},\FF_\ell)$;  for background on \'etale cohomology, see \cite{MR559531}, \cite{MR926276} or \cite{MR0463174}.   There is a natural action of $G_K$ on $V$ that we can express in terms of a representation
\[
\rho\colon G_K \to \Aut(V).
\]
Let $\FF_\ell(1)$ be the group of $\ell$-th roots of unity in $\Kbar$ and let $\FF_\ell(-1)$ be the $\FF_\ell$-dual of $\FF_\ell(1)$.  Evaluation gives a natural isomorphism $\FF_\ell(1) \otimes_{\FF_\ell} \FF_\ell(-1) = \FF_\ell$.

One knows that $J[\ell]$ is isomorphic to the \'etale cohomology group $H^1(C_{\Kbar},\FF_\ell(1))$ as an $\FF_\ell[G_K]$-module.  The group $G_K$ acts trivially on $\FF_\ell(1)$ since $\ell\neq p$, so $J[\ell]$ and $V$ are isomorphism $\FF_\ell[G_K]$-modules.   Therefore, $\rho$ and $\rho_{J,\ell}|_{G_K}$ are isomorphic representations.  It thus suffices to prove that $\rho(G_K)$ is a group of order $\ell$ and that it is generated by a transvection when $p\in \{7,11\}$.  

 We have an alternating pairing 
\[
\ang{\,,\,}\colon V \times V \xrightarrow{\cup} H^2(C_{\Kbar},\FF_\ell\otimes_{\FF_\ell} \FF_\ell) = H^2(C_{\Kbar},\FF_\ell) \xrightarrow{\sim} \FF_\ell(-1),
\]
where we are composing the cup product and trace map.  The $G_K$-action on $V$ respects the pairing $\ang{\,,\,}$, i.e., $\ang{\sigma(v),\sigma(w)}=\sigma(\ang{v,w})=\ang{v,w}$ for all $\sigma\in G_K$ and $v,w\in V$.   The pairing $\ang{\,,\,}$ is non-degenerate (after taking Tate twists, we can identify this pairing with the Weil pairing on $J[\ell]$).  \\

 The morphism $\calC_{R} \to \Spec R$ is a proper and flat morphism of relative dimension $1$.   From Lemma~\ref{L:first computation}(\ref{L:first computation ii}), we know that  $\calC_K$ is smooth and that $\calC_{\FFbar_p}$ is smooth away from a finite set $\Sigma$ of ordinary double points.

 The \defi{Picard-Lefschetz formula} (see \cite{SGA7-2}*{XV~Th\'eor\`eme~3.4})  shows that there are non-zero and pairwise orthogonal \emph{vanishing cycles} $\{\delta_x\}_{x\in \Sigma}$ in $V$ such that for $v\in V$ and $\sigma\in G_K$, we have
\[
\sigma(v) = v - \sum_{x\in \Sigma} \varepsilon_x(\sigma) \ang{v,\delta_x}\cdot \delta_x,
\]
where $\varepsilon_x \colon G_K \to \FF_\ell(1)$ is a certain homomorphism  and where we view $\varepsilon_x(\sigma) \ang{v,\delta_x}$ as an element of $\FF_\ell(1) \otimes_{\FF_\ell} \FF_\ell(-1) = \FF_\ell$.

Let $\varepsilon \colon G_K \to \FF_\ell(1)$ be the surjective homomorphism that satisfies $\sigma(\sqrt[\ell]{p})= \varepsilon(\sigma) \sqrt[\ell]{p}$ for all $\sigma\in G_K$.   From Lemma~\ref{L:first computation}(\ref{L:first computation ii}), we find that the completion of the local ring of $\calC_{R}$ at a point $x\in \Sigma$ is isomorphic as an $R$-algebra to $R[[x,y]]/(xy + p)$.   From \cite{SGA7-2}*{XV~\S3.3}, we find that $\varepsilon_x=\varepsilon$ (in general, you would need to raise $\varepsilon$ to some power).   Therefore,
\begin{align} \label{E:PL}
\sigma(v) = v - \sum_{x\in \Sigma} \varepsilon(\sigma) \ang{v,\delta_x}\cdot \delta_x
\end{align}
for $v\in V$ and $\sigma\in G_K$.   The representation $\rho$ thus factors through the order $\ell$ group $\Gal(K(\sqrt[\ell]{p})/K)$.  Therefore, $\rho(G_K)$ is a group of order $1$ or $\ell$.  Since $\varepsilon\neq 1$ and $v\mapsto \ang{v,\delta_x}$ is non-trivial, we deduce from (\ref{E:PL}) that $\rho(G_K)$  is a non-trivial group  and hence is a cyclic of order $\ell$.

Now suppose that $p\in\{7,11\}$.  Fix any $\sigma_0 \in G_K$ with $\rho(\sigma_0)\neq 1$.  It remains to prove that $\rho(\sigma_0)$ is a transvection.  Since $\rho(\sigma_0)\neq 1$ has order $\ell$, it suffices to prove that $\sigma_0$ fixes an $\FF_\ell$-subspace of $V$ of dimension $\dim_{\FF_\ell} V -1$.  Since $p\in \{7,11\}$, we have $|\Sigma|=1$ by Lemma~\ref{L:first computation}(\ref{L:first computation ii}).   We thus have 
\begin{align} \label{E:PL 2}
\sigma_0(v) = v - \varepsilon(\sigma_0) \ang{v,\delta_x}\cdot \delta_x,
\end{align}
where $x$ is the unique element of $\Sigma$.  Let $W$ be the subspace of $V$ consisting of $v\in V$ for which $\ang{v,\delta_x}$ is trivial; it has $\FF_\ell$-dimension $\dim_{\FF_\ell} V - 1$ since the pairing is non-degenerate and $\delta_x\neq 0$.   By (\ref{E:PL 2}), we deduce that $\sigma_0(v)=v$ for all $v\in W$.  This completes the proof of Proposition~\ref{P:bad primes}.

\section{Inertia at $\ell$} \label{S:inertia at ell}

Fix an odd prime $\ell \notin S$ and let $I_\ell$ be any inertia subgroup of $G_\QQ=\Gal(\Qbar/\QQ)$ at the prime $\ell$.   In this section, we give some information on how $I_\ell$ acts on $J[\ell]$.\\

We first need to recall some background on tame inertia groups and tame inertia weights,  see \S1 of \cite{MR0387283} for more details.  Let $\calP \subseteq I_\ell$ be the \defi{wild inertia subgroup} of $I_\ell$; it is the largest pro-$\ell$ subgroup of $I_\ell$.   The quotient $I_{\ell}^t:= I_\ell /\calP$ is the \defi{tame inertia group} for the prime $\ell$.   For an integer $d\geq 1$ relatively prime to $\ell$, let $\mu_d$ be the $d$-th roots of unity in $\Qbar$.   The map 
\[
\theta_{d}\colon I_\ell \to \mu_d,\quad \sigma\mapsto \sigma(\sqrt[d]{\ell})/\sqrt[d]{\ell}
\]
is a surjective homomorphism which factors through $I^t_\ell$.   Taking the inverse limit over all $d$ relatively prime to $\ell$ (ordered by divisibility), we obtain an isomorphism 
\[
I_\ell^t \xrightarrow{\sim} {\varprojlim}_d \, \mu_d.
\]   
By composing the homomorphism $\theta_d$ with reduction modulo a place of $\Qbar$ lying over $\ell$, we obtain a character $I_\ell^t \to \FFbar_\ell^\times$.   For an integer $m\geq 1$, setting $d:=\ell^m-1$ gives a surjective character 
\[
\phi\colon I_\ell^t \to \FF_{\ell^m}^\times.
\] 
The \defi{fundamental characters} of level $m$ are the $m$ characters $I_\ell^t \to \FF_{\ell^m}^\times$ obtained by composing $\phi$ with the isomorphisms of $\FF_{\ell^m}^\times$ arising from field automorphisms of $\FF_{\ell^m}$; they are $\phi$, $\phi^\ell, \ldots, \phi^{\ell^{m-1}}$.\\

Let $V$ be an irreducible $\FF_\ell[I_\ell]$-module and set $m:=\dim_{\FF_\ell}(V)$.  There is then an isomorphism $V\cong \FF_{\ell^m}$ of $\FF_\ell$-vector spaces such that the induced character 
\[
I_\ell \to \Aut_{\FF_\ell}(V) \cong \Aut_{\FF_\ell}(\FF_{\ell^m})
\]
has image in $\FF_{\ell^m}^\times$, where the first map gives the action of $I_\ell$ on $V$ (here we use scalar multiplication to identify $\FF_{\ell^m}^\times$ with a subgroup of $\Aut_{\FF_\ell}(\FF_{\ell^m})$).  This representation arising from $V$ thus factors through a character $\alpha\colon I_\ell^t \to \FF_{\ell^m}^\times$.    Given a fundamental character $\phi\colon I_\ell^t\to \FF_{\ell^m}^\times$ of level $m$, there are unique integers $0\leq e_1,\ldots, e_m \leq \ell-1$ such that 
\begin{equation} \label{E:amplitude}
\alpha= \phi^{e_1+e_2\ell + \cdots + e_m \ell^{m-1}}.
\end{equation}
The integers $e_1,\ldots, e_m$ are called the \defi{tame inertia weights} of $V$.

Let $V$ be an $\FF_\ell[I_\ell]$-module with $V$ a finite dimensional $\FF_\ell$-vector space.  Let $V_1,\ldots, V_r$ be the composition factors of $V$ as an $\FF_\ell[I_\ell]$-module.   An integer is a \defi{tame inertia weight} for $V$ if it is a tame inertia weight for at least one of the $V_i$.

We now consider the representations occurring in this paper. 

\begin{prop} \label{P:amplitude}
Fix an odd prime $\ell \notin S$.  The only possible tame inertia weights for the $\FF_\ell[I_\ell]$-module $J[\ell]$ are $0$ and $1$.
\end{prop}
\begin{proof}
This follows from work of Raynaud, cf.~\cite{MR0419467}*{Corollaire~3.4.4}.  One could also deduce this from \cite{MR2372809}.
\end{proof}

The following lemma gives some consequences of Proposition~\ref{P:amplitude} that we will use later.   Recall that $\chi_\ell \colon G_\QQ \to \FF_\ell^\times$ was defined in \S\ref{SS:background}.

\begin{lemma} \label{L:amplitude cor}
Fix an odd prime $\ell\notin S$ and let $V$ be any composition factor of the $\FF_\ell[I_\ell]$-module $J[\ell]$.    Set $m:=\dim_{\FF_\ell} V$ and let $\rho\colon I_\ell \to \Aut_{\FF_\ell}(V)$ be the representation describing the $I_\ell$-action.
\begin{romanenum}
\item \label{L:amplitude cor i}
We have $\det \circ \rho = \chi_\ell^e |_{I_\ell}$ for some integer $0\leq e \leq m$.
\item \label{L:amplitude cor ii}
If $H$ is a closed subgroup of $I_\ell$ satisfying $[I_\ell:H]<\ell-1$, then $\rho|_H$ is irreducible.
\end{romanenum}
\end{lemma}
\begin{proof}
We first prove (\ref{L:amplitude cor i}).  As noted above, $\rho$ gives rise to a character $\alpha\colon I_\ell^t \to \FF_{\ell^m}^\times$ of the form (\ref{E:amplitude}) with $\phi$ a fundamental character of level $m$ and $0\leq e_1,\ldots, e_m \leq \ell-1$.  By Proposition~\ref{P:amplitude}, we have $e_1,\ldots, e_m \in \{0,1\}$.

The character $\det \circ \rho \colon I_\ell \to \FF_\ell^\times$ factors through $N\circ \alpha \colon I_\ell^t \to \FF_\ell^\times$, where $N:=N_{\FF_{\ell^m}/\FF_\ell}\colon \FF_{\ell^m}^\times \to \FF_\ell^\times$ is the norm map.   Therefore,
\[
N \circ \alpha = (N \circ \phi)^{e_1+e_2\ell + \cdots + e_m \ell^{m-1}}= (N\circ \phi)^{e_1+\ldots+e_m},
\]
where we have used that $N^{\ell} = N$.   We have $0\leq e_1+\ldots + e_m \leq m$, so it suffices to prove that $N\circ \phi = \chi_\ell |_{I_\ell}$.  We have $N\circ \phi = \phi^{1+\ell +\ldots +\ell^{m-1}}$ which one can check is the (unique) fundamental character of level $1$.   Part (\ref{L:amplitude cor i}) follows by noting that the fundamental character of level $1$ is $ \chi_\ell |_{I_\ell}$, cf.~\cite{MR0387283}*{Prop.~8}.

We now prove (\ref{L:amplitude cor ii}).  Set $T:=\FF_{\ell^m}^\times$ and define the representation 
\[
\beta\colon T \to \FF_{\ell^m}^\times \subseteq \Aut_{\FF_\ell}(\FF_{\ell^m}),\quad x\mapsto x^{e_1+e_2\ell +\cdots + e_m\ell^{m-1}}.
\]
The representation $\rho$ is isomorphic to the one obtained by composing the surjective character $I_\ell \to I_\ell^t \xrightarrow{\phi} T$ with $\beta$.  The representation $\beta$ is irreducible since $V$ is an irreducible $\FF_\ell[I_\ell]$-module.

Let $H$ be any closed subgroup of $I_\ell$ satisfying $[I_\ell:H] < \ell-1$.  The subgroup $S:= \phi(H)$ of $T$ then satisfies $[T:S] < \ell-1$.

We can now use the rigidity of tori as described by Hall in \cite{MR2820155}.   In the language of \cite{MR2820155}*{\S2}, the \defi{amplitude} of $\beta$ is $\max\{e_i\}$ which in our case is $0$ or $1$.   Lemma~3 of \cite{MR2820155} and our condition $[T:S]<\ell-1$ implies that $\beta(T)$ and $\beta(S)$ have the same centralizer in $\Aut_{\FF_\ell}(\FF_{\ell^m})$. 

Suppose that $\beta|_S \colon S \to \Aut_{\FF_\ell}(\FF_{\ell^m})$ is reducible.   Since $\ell \nmid |S|$, we have $\FF_{\ell^m} = W_1 \oplus W_2$, where $W_1$ and $W_2$ are non-zero $\FF_\ell$-subspaces fixed under the action of $S$.  Take $A\in \Aut_{\FF_\ell}(\FF_{\ell^m})$ such that $A(w)=w$ for $w\in W_1$ and $A(w)=-w$ for $w\in W_2$.   Since $A$ commutes with $\beta(S)$, we deduce that it also commutes with $\beta(T)$.   For any $B\in \beta(T)$ and $w\in W_1$, we have $A(Bw)= B(Aw)=Bw$.     Therefore, $B(W_1)\subseteq W_1$ for all $B\in \beta(T)$ which contradicts that $\beta$ is irreducible.   

Therefore, the representation $\beta|_S$, and hence also $\rho|_H$, is irreducible.
\end{proof}

\section{Irreducibility}  \label{S:irreducible}

For a prime $\ell$, the $\FF_\ell$-vector space $J[\ell]$ has dimension $6$ and comes with a natural $G_\QQ$-action.  The goal of this section is to prove the following:

\begin{prop} \label{P:irreducible}
For every odd prime $\ell$, the $\FF_\ell[G_\QQ]$-module $J[\ell]$ is irreducible.
\end{prop}

Suppose that there is an odd prime $\ell$ such that $J[\ell]$ is a reducible $\FF_\ell[G_\QQ]$-module; we will try to obtain a contradiction. We first exclude a few possibilities for $\ell$.

\begin{lemma} \label{L:excluded primes}
We have $\ell \notin \{3,5,7,11,41,83\}$.
\end{lemma}
\begin{proof}
For one of the given primes $\ell \in \{3,5,7,11,41,83\}$, it suffices to show that there is a prime $p\notin S\cup\{\ell\}$ such that $P_p(T)$ is irreducible modulo $\ell$.    

We have computed $P_p(T)$ for several small $p$, cf.~\S\ref{S:good}.  The polynomial $P_{17}(x)$ is irreducible modulo $3$.   The polynomial $P_{41}(x)$ is irreducible modulo $5$.  The polynomial $P_{2}(x)$ is irreducible modulo $7$, $11$ and $41$.  The polynomial $P_{19}(x)$ is irreducible modulo $83$.
\end{proof}

For the rest of this section, we may thus assume that $\ell$ is odd and $\ell \notin \{3,5,7,11,41, 83\}$.  In particular, $\ell \notin S$.

Let $V_1,\ldots, V_r$ be the composition factors of $J[\ell]$ as an $\FF_\ell[G_\QQ]$-module; the semisimplification of $J[\ell]$ as an $\FF_\ell[G_\QQ]$-module is then isomorphic to $V_1\oplus \cdots \oplus V_r$.    We have $r\geq 2$ since $J[\ell]$ is a reducible $\FF_\ell[G_\QQ]$-module by assumption.

Let 
\[
\rho_i \colon G_\QQ \to \Aut_{\FF_\ell}(V_i)
\]
be the Galois representation corresponding to $V_i$.    Define $d_i = \dim_{\FF_\ell} V_i$.   We may assume that the $V_i$ have been numbered so that $d_1 \leq \cdots \leq d_r$.    We have $\sum_i d_i = 6$, so $d_1 \in \{1,2,3,6\}$.     We have $r\geq 2$, so $d_1 \in \{1,2,3\}$.

 We will rule out the three cases  $d_1 \in \{1,2,3\}$ separately in \S\S\ref{SS:1d}--\ref{SS:3d}.   This contradiction will imply that  $J[\ell]=V_1$ is an irreducible $\FF_\ell[G_\QQ]$-module.

\subsection{Determinants} \label{SS:determinants}

Fix a finite dimensional $\FF_\ell$-vector space $W$ with an action of $G_\QQ$ given by a representation $\rho\colon G_\QQ \to \Aut_{\FF_\ell}(W)$.  Let $W^\vee$ to be the dual vector space of $W$ and let $\rho^* \colon G_\QQ \to \Aut_{\FF_\ell}(W^\vee)$ be the contragredient representation,  i.e., $\rho^*(\sigma)$ is the transpose of $\rho(\sigma^{-1})$.    Let $W^\vee(1)$ be the vector space $W^\vee$ where $G_\QQ$ acts via the representation $\chi_\ell \cdot \rho^*$.

Since the pairing $J[\ell]\times J[\ell] \to \mu_\ell$ coming from the Weil pairing and the natural principal polarization of $J$ is non-degenerate, we find that $J[\ell]$ and $J[\ell]^\vee(1)$ are isomorphic $\FF_\ell[G_\QQ]$-modules.   Therefore, the $\FF_\ell[G_\QQ]$-modules $V_1,\ldots, V_r$ are isomorphic to $V_1^\vee(1),\ldots, V_r^\vee(1)$, though possibly in a different order.

The following lemma constrains the possibilities for the characters $\det\circ \rho_i \colon G_\QQ \to \FF_\ell^\times$.

\begin{lemma} \label{L:determinant of factors}
\begin{romanenum}
\item \label{L:determinant of factors i}
For each $1\leq i \leq r$, there is a unique integer $0\leq e_i \leq d_i$ such that $\det\circ \rho_i = \chi_\ell^{e_i}$.
\item \label{L:determinant of factors ii}
We have $\sum_{i=1}^r e_i = 3$.
\item \label{L:determinant of factors iii}
We have $\{e_1,\ldots, e_r\} = \{d_1-e_1,\ldots, d_r-e_r\}$.
\item \label{L:determinant of factors iv}
If $V_i^\vee(1) \cong V_i$, then $d_i$ is even and $e_i = d_i/2$.
\end{romanenum}
\end{lemma}
\begin{proof} 
Fix an integer $1\leq i \leq r$.  The semi-simplification of $V_i$ as an $\FF_\ell[I_\ell]$-module is of the form $W_{i,1} \oplus \cdots \oplus W_{i,s}$, where $W_{i,j}$ is an irreducible $\FF_\ell[I_\ell]$-module.   By Lemma~\ref{L:amplitude cor}(\ref{L:amplitude cor i}), the determinant of the action of $I_\ell$ on $W_{i,j}$ is a character $I_\ell \to \FF_\ell^\times$ equal to $\chi_\ell^{e_{i,j}} |_{I_\ell}$ for some integer $0\leq e_{i,j} \leq \dim_{\FF_\ell} W_{i,j}$.    Therefore, 
\begin{equation} \label{E:alpha unramified}
(\det\circ \rho_i) |_{I_\ell} = \prod_{j=1}^s \chi_\ell^{e_{i,j}}|_{I_\ell} = \chi_\ell^{e_i} |_{I_\ell},
\end{equation}
where $e_i := \sum_{j=1}^s e_{i,j}$.   We have $0 \leq e_i \leq \sum_{j=1}^s \dim_{\FF_\ell} W_{i,j} = \dim_{\FF_\ell} V_i=d_i$.   Define the character
\[
\alpha_i := (\det\circ \rho_i) \cdot \chi_\ell^{-e_i} \colon G_\QQ \to \FF_\ell^\times.
\]
The representation $\rho_{J,\ell}$, and hence also $\alpha_i$, is unramified at all primes $p\notin S \cup \{\ell\}$.   Since the order of $\FF_\ell^\times$ is relatively prime to $\ell$, Proposition~\ref{P:bad primes} implies that $\alpha_i$ is also unramified at the primes $p\in S$.   The character $\alpha_i$ is unramified at the prime $\ell$ by (\ref{E:alpha unramified}).   We thus have $\alpha_i=1$ since $\alpha_i \colon G_\QQ \to \FF_\ell^\times$ is unramified at all primes and $\QQ$ has no non-trivial extensions unramified at all primes.  Therefore, $\det\circ \rho_i = \chi_\ell^{e_i}$.   

This proves the existence  of $e_i$ in (\ref{L:determinant of factors i}); it remains to prove the uniqueness.     Take any integer $0\leq f \leq d_i$ such that $\det\circ \rho_i = \chi_\ell^f$.   We thus have $\chi_\ell^{f-e_i}=1$ and hence $f-e_i \equiv 0 \pmod{\ell-1}$ since $\chi_\ell$ has order $\ell-1$.   We have $d_i \leq 3$ since $r\geq 2$, so $|f-e_i|\leq 3$.   Since $|f-e_i| \leq 3 <\ell-1$ and $f-e_i\equiv 0 \pmod{\ell-1}$, we must have $f=e_i$.  

We now prove part (\ref{L:determinant of factors ii}).  Since $\det(\rho_{J,\ell}(\Frob_p))\equiv P_p(0)=p^3 \pmod{\ell}$ for all $p\notin S \cup \{\ell\}$, we have $\det \circ \rho_{J,\ell} = \chi_\ell^3$.  Therefore, $\chi_\ell^3 = \prod_{i=1}^r \det \circ \rho_i = \chi_\ell^{e}$, where $e:=\sum_{i=1}^r e_i$.   We have $3-e\equiv 0 \pmod{\ell-1}$ since $\chi_\ell^{3-e}=1$ and $\chi_\ell$ has order $\ell-1$.   We have $|3-e| \leq 3$ since $0\leq e \leq \sum_i d_i = 6$.   Since $3-e \equiv 0 \pmod{\ell-1}$ and $|3-e|\leq 3 < \ell-1$, we conclude that $e=3$ which proves (\ref{L:determinant of factors ii}).

We now prove part (\ref{L:determinant of factors iii}).  Fix an integer $1\leq i \leq r$.   The representation $G_\QQ \to \Aut_{\FF_\ell}(V_i^\vee)$, $\sigma\mapsto \chi_\ell(\sigma)\cdot \rho_i^*(\sigma)$ is isomorphic to $V_i^\vee(1)$ and its determinant is given by 
\[
G_\QQ\to \FF_\ell^\times, \quad \sigma \mapsto \det (\chi_\ell(\sigma) \rho_i^*(\sigma)) = \chi_\ell(\sigma)^{d_i} \, \det( \rho_i(\sigma^{-1})) = \chi_\ell(\sigma)^{d_i - e_i}.
\]
We noted above that the $\FF_\ell[G_\QQ]$-modules $V_1,\ldots, V_r$ are isomorphic to $V_1^\vee(1),\ldots, V_r^\vee(1)$ though possibly in a different order.   Therefore, $\{\chi_\ell^{e_1},\ldots, \chi_\ell^{e_r}\} = \{\chi_\ell^{d_1-e_1},\ldots, \chi_\ell^{d_r-e_r}\}$.  The uniqueness in part (\ref{L:determinant of factors i}) implies that $\{e_1,\ldots, e_r\}=\{d_1-e_1,\ldots, d_r-e_r\}$.  

It remains to prove part (\ref{L:determinant of factors iv}).  If $V_i\cong V_i^\vee(1)$, then the computation above shows that $\det \circ \rho_i$ is equal to both $\chi_\ell^{e_i}$ and $\chi_\ell^{d_i-e_i}$.   Therefore, $e_i=d_i-e_i$ by the uniqueness in part (\ref{L:determinant of factors i}) and hence $d_i=2e_i$.
\end{proof}

\subsection{One-dimensional case}  \label{SS:1d}
Suppose that $d_1 = 1$.  The Galois action on $V_1$ is described by the character $\rho_1\colon G_\QQ \to \Aut_{\FF_\ell}(V_1) =\FF_\ell^\times$.  So for any prime $p\notin S \cup \{\ell\}$, $\rho_1(\Frob_p)$ is a root of $P_p(T)\equiv \det(TI-\rho_{J,\ell}(\Frob_p)) \pmod{\ell}$.     The character $\rho_1$ is $1$ or $\chi_\ell$ by Lemma~\ref{L:determinant of factors}(\ref{L:determinant of factors i}), so $P_p(1) \equiv 0 \pmod{\ell}$ or $P_p(p)\equiv 0 \pmod{\ell}$.

With $p=2$ and using the polynomial $P_2(T)$ from \S\ref{S:good}, we find that $P_2(1)=3\cdot 17$ and $P_2(2)=2^3\cdot 3\cdot 17$.   Therefore, $\ell \in \{3,17\}$.    However, this contradicts Lemma~\ref{L:excluded primes}.

This completes our proof that the case $d_1=1$ does not occur.

\subsection{Two-dimensional case} \label{SS:2d}
Suppose that $d_1=2$.  

\begin{lemma} \label{L:correct 2d}
We have $d_i=2$ and $\det \circ \rho_i = \chi_\ell$ for some $1\leq i \leq r$.
\end{lemma}
\begin{proof}
Since $d_1=2$, we either have $r=3$ with $d_1=d_2=d_3=2$ or $r=2$ with $d_1=2$ and $d_2=4$.   If $r=3$, then Lemma~\ref{L:determinant of factors}(\ref{L:determinant of factors i}) and (\ref{L:determinant of factors ii}) imply that $e_i=1$ for some $1\leq i \leq 3$ and hence $d_i=2$ and $\det\circ \rho_i = \chi_\ell$.  So suppose that $r=2$ and hence $(d_1,d_2)=(2,4)$.  Lemma~\ref{L:determinant of factors}(\ref{L:determinant of factors iii}) implies that $\{e_1,e_2\}=\{2-e_1,4-e_2\}$ and this can only hold if $e_1=1$ and $e_2=2$.  Therefore, $d_1=2$ and $\det\circ \rho_1 = \chi_\ell$.
\end{proof}

After possibly renumbering the $V_i$, we may assume by Lemma~\ref{L:correct 2d} that 
\[
\rho_1 \colon G_\QQ \to \Aut_{\FF_\ell}(V_1)\cong \GL_2(\FF_\ell)
\] 
has determinant $\chi_\ell$.    The following lemma uses Serre's conjecture to relate $\rho_1$ to a newform of weight $2$ and bounded level.  

\begin{lemma} \label{L:Serre conjecture}
There exists a newform $f=q+\sum_{n\geq 2} a_n(f) q^n \in S_2(\Gamma_0(N))$ with $N$ dividing $7\cdot 11\cdot 83=6391$ and a maximal ideal $\lambda$ of the ring of integers of the number field $\QQ(a_n(f))$ such that
\[
\tr(\rho_1(\Frob_p)) \equiv a_p(f)\pmod{\lambda}
\] 
for all primes $p \notin S\cup{\ell}$.
\end{lemma}
\begin{proof}
The $2$-dimensional representation $\rho_1$ is irreducible and is also odd since $\det\rho_1 = \chi_\ell$.  \defi{Serre's conjecture} \cite{MR885783}, proved by Khare and Wintenberger \cites{MR2551763,MR2551764}, implies that the representation $\rho_1$ is isomorphic to one arising from some newform $f$.  Moreover, the newform $f=q+\sum_{n\geq 2} a_n(f) q^n$ can be found in $S_k(\Gamma_1(N))$ with prescribed weight $k$ and level $N$.  Let $K$ be the subfield of $\CC$ generated by the Fourier coefficients of $f$; it is a number field.  Let $\varepsilon\colon (\ZZ/N\ZZ)^\times \to K^\times$ be the nebentypus of $f$.    There is thus a maximal ideal $\lambda$ of the ring of integers of $K$ such that
\[
\det(xI - \rho_1(\Frob_p)) \equiv x^2 - a_p(f) x + \varepsilon(p) p^{k-1} \pmod{\lambda}
\]
for all primes $p\nmid N\ell$.  \\

Let us compute the weight $k$.   Suppose that $\rho_1|_{I_\ell}$ is reducible.  The semisimplification of $\rho_1|_{I_\ell}$ is then given by two characters $\varphi_1,\varphi_2 \colon I_\ell^t \to \FF_\ell^\times$.   By Proposition~\ref{P:amplitude}, each $\varphi_i$ is either $1$ or the fundamental character of level $1$.  Since the fundamental character of level $1$ is $ \chi_\ell |_{I_\ell}$, cf.~\cite{MR0387283}*{Prop.~8} and $\det \circ \rho_1 = \chi_\ell$, we deduce that $\{\varphi_1,\varphi_2\}=\{1,\chi_\ell |_{I_\ell}\}$.  In the notation of \S2.3 of \cite{MR885783}, we have $a=0$ and $b=1$, and hence $k=1+\ell a + b =2$.

Now suppose that $\rho_1|_{I_\ell}$ is irreducible.  As explained in \S\ref{S:inertia at ell} (and using Proposition~\ref{P:amplitude}), $\rho_1|_{I_\ell}$ factor through $I_\ell^t$ and is then isomorphic to a representation of the form
\[
I_\ell^t \xrightarrow{\phi^{e_1 + e_2 \ell}} \FF_{\ell^2}^\times \subseteq \Aut_{\FF_\ell}(\FF_{\ell^2}),
\]
where $\phi\colon I_\ell^t \to \FF_{\ell^2}^\times$ is a fundamental character of level $2$ and $0\leq e_1,e_2 \leq 1$.  We have $\{e_1,e_2\}=\{0,1\}$ since otherwise $\phi^{e_1 + e_2 \ell}$ would have image in $\FF_\ell^\times$ which would contradict the irreducibility of $\rho_1|_{I_\ell}$.   Therefore, the $I_\ell^t$-action on $V_1\otimes_{\FF_\ell} \FF_{\ell^2}$ is diagonalizable and is given by the characters $\phi,\phi^\ell \colon I_\ell^t\to \FF_{\ell^2}^\times$.      In the notation of \S2.2 of \cite{MR885783}, we may take $a=0$ and $b=1$, and hence $k=1+\ell a + b =2$.
\\

We now consider the level $N$.  The representation $\rho_{J,\ell}$, and hence also $\rho_1$, is unramified at all primes $p\notin S\cup \{\ell\}$.    Take any $p\in S$; we have $p\neq \ell$ by Lemma~\ref{L:excluded primes}.  Let $V_1^{I_p}$ be the subspace of $V_1$ fixed by $I_p$.     From \cite{MR885783}*{\S1.2}, the Artin conductor of $\rho_1$ is  $N := \prod_{p\in S} p^{n_p}$ where $n_p = \dim_{\FF_\ell} V_1/V_1^{I_p}$; there is no wild ramification since the cardinality of $\rho_{J,\ell}(I_p)$ is not divisible by $p$ by Lemma~\ref{P:bad primes}.   Since $\rho_{J,\ell}(I_p)$ is a group of order $\ell$ by Lemma~\ref{P:bad primes}, the group $\rho_1(I_p)$ has order $1$ or $\ell$.  If $\rho_1(I_p)$ has order $1$, then $n_p= 0$.   If $\rho_1(I_p)$ has order $\ell$, then it is conjugate in $\Aut_{\FF_\ell}(V_1) \cong \GL_2(\FF_\ell)$ to the group generated by $\left(\begin{smallmatrix}1 & 1 \\  0 & 1\end{smallmatrix}\right)$.  In this last case, we have $n_p=1$.   This completes the proof that $N$ divides $7\cdot 11 \cdot 83 = 6391$.\\

Finally, it remains to show that $f \in S_2(\Gamma_1(N))$ actually lies in $S_2(\Gamma_0(N))$; equivalently, that the nebentypus $\varepsilon$ is trivial.  Let $\mu$ be the image of $\varepsilon$; it is a finite group of roots of unity in $K$.  

With $\OO$ the ring of integers of $K$, the kernel of the reduction modulo $\lambda$ homomorphism $\mu \to (\OO/\lambda)^\times$ is an $\ell$-group.     For any $p\notin S\cup\{\ell\}$,  the equality $\det \circ \rho_1 = \chi_\ell$ implies that $\varepsilon(p) p \equiv  \chi_\ell(p)= p \pmod{\lambda}$ and hence $\varepsilon(p)\equiv 1 \pmod{\lambda}$.    Therefore, $\mu$ is an $\ell$-group.  Since $|\mu|$ divides the cardinality of $(\ZZ/N\ZZ)^\times \cong (\ZZ/7\ZZ)^\times \times (\ZZ/11\ZZ)^\times \times (\ZZ/83 \ZZ)^\times$, we deduce that $\mu=1$ or $\ell \in \{3, 5, 41\}$.   We have $\ell \notin \{3, 5, 41\}$ by Lemma~\ref{L:excluded primes}, so $\mu=1$ and hence $\varepsilon=1$.
\end{proof}

Take any prime $p\notin S \cup \{\ell\}$.   

Let $H_p(x)$ be the characteristic polynomial of the Hecke operator $T_p$ acting on $S_2(\Gamma_0(6391))$; it is monic with integer coefficients.
Take $f$ and $\lambda$ as in Lemma~\ref{L:Serre conjecture}.   Since $p \nmid 6391$, there is a cusp form $f' \in S_2(\Gamma_0(6391))$ such that $T_p(f') = a_p(f) f'$;  we can take $f'$ to be an oldform if $N$ properly divides $6391$.    Therefore, $H_p(a_p(f)) = 0$ and in particular $H_p(a_p(f))\equiv 0\pmod{\lambda}$.  Lemma~\ref{L:Serre conjecture} then implies that $\tr(\rho_1(\Frob_p))\in \FF_\ell$ is a root of $H_p(x)$.\\

Let $P_p(T)$ be the polynomial from \S\ref{S:good}.     Define the polynomial $Q_p(x) := {\prod}_{\alpha} (x - \alpha)$,  where $\alpha$ runs over the values $\lambda+p/\lambda$ with $\lambda\in \Qbar$ being a root of $P_p(x)$.         The polynomial $Q_p(x)$ is monic with integer coefficients.   Since $\det(\rho_1(\Frob_p))=\chi_\ell(\Frob_p)\equiv p \pmod{\ell}$, we have $\tr(\rho_1(\Frob_p)) = \lambda + p/\lambda$ for some root $\lambda\in \FFbar_\ell$ of $P_p(T)$.  Therefore, $\tr(\rho_1(\Frob_p)) \in \FF_\ell$ is a root of $Q_p(x)$ modulo $\ell$.     

A computation show that $Q_2(x)= x^3 + 3x^2 - 3$ for $Q_5(x)=x^3+4x^2-5x-23$.   For example, the following code gives $Q_2(x)$; one could also compute  $Q_2(x)$ using approximations for the roots of $P_2(T)$ in $\CC$ and use that $Q_2(x)$ has integer coefficients.  

{\small
\begin{verbatimtab}
	_<T>:=PolynomialRing(Rationals());
	p:=2;  P:=T^6+3*T^5+6*T^4+9*T^3+12*T^2+12*T+8;
	K:=SplittingField(P); Pol<x>:=PolynomialRing(K);
	&*[x-a : a in {r[1]+p/r[1]: r in Roots(Pol!P)}];
\end{verbatimtab}
}

Let $r_p$ be the resultant of $H_p(x)$ and $Q_p(x)$; it is an integer.   Since $\tr(\rho_1(\Frob_p)) \in \FF_\ell$ is a common root of $H_p(x)$ and $Q_p(x)$, we deduce that $\ell$ divides $r_p$.   \\ 

The \texttt{Magma} code below shows that the greatest common divisor of $r_2$ and $r_5$ is $3^{16}$. 
{\small
\begin{verbatimtab}
	Pol<x>:=PolynomialRing(Rationals());
	S:=CuspForms(Gamma0(7*11*83),2);
	H2:=Pol!HeckePolynomial(S,2);
	H5:=Pol!HeckePolynomial(S,5);
	r2:=Integers()!Resultant(H2,x^3+3*x^2-3);
	r5:=Integers()!Resultant(H5,x^3+4*x^2-5*x-23);
	GCD([r2,r5]) eq 3^16;
\end{verbatimtab}
}
Since $\ell$ divides $r_2$ and $r_5$, we must have $\ell=3$.    However, this is impossible by Lemma~\ref{L:excluded primes}.

This shows that the case $d_1=2$ does not occur.

\begin{remark}
To compute the Hecke polynomials, one could also use modular symbols (in our case, \texttt{Magma} does this approach much faster).   For example, one can compute $H_2(x)$ by the code:
{\small
\begin{verbatimtab}
	M:=CuspidalSubspace(ModularSymbols(7*11*83,2,1));
	CharacteristicPolynomial(HeckeOperator(M,2));
\end{verbatimtab}
}
\noindent We are not using $p=3$ in the above computations because $r_3=0$.
\end{remark}

\subsection{Three-dimensional case}  \label{SS:3d}

Suppose that $d_1=3$, and hence $r=2$ with $d_1=d_2=3$.   After possibly swapping $V_1$ and $V_2$, we may assume by Lemma~\ref{L:determinant of factors} that there is an integer $e\in \{0,1\}$ such that $\det \circ \rho_1 = \chi_\ell^{e}$.  

\begin{lemma} \label{L:dim 3}
Take any prime $p\notin S\cup \{\ell\}$.  If $\alpha,\beta, \gamma \in \FFbar_\ell$ are the roots of $\det(xI- \rho_1(\Frob_p)) \in \FF_\ell[x]$, then $p/\alpha,p/\beta, p/\gamma \in \FFbar_\ell$ are the roots of $\det(xI - \rho_2(\Frob_p))$.
\end{lemma}
\begin{proof}
With notation as in the beginning of \S\ref{SS:determinants}, the roots of the characteristic polynomial of $\rho_1^*(\Frob_p)$ in $\FF_\ell[x]$ are $1/\alpha$, $1/\beta$ and $1/\gamma$.   Therefore, the roots of the characteristic polynomial of $\chi_\ell(\Frob_p) \rho_1^*(\Frob_p) = p \rho_1^*(\Frob_p)$ are $p/\alpha$, $p/\beta$ and $p/\gamma$.  It thus suffices to show that $V_2$ and $V_1^\vee(1)$ are isomorphic $\FF_\ell[G_\QQ]$-modules.  The $\FF_\ell[G]$-modules $V_1^\vee(1)$ is isomorphic to $V_1$ or $V_2$.   By Lemma~\ref{L:determinant of factors}(\ref{L:determinant of factors iv}), we must have $V_2\cong V_1^\vee(1)$.
\end{proof}

Take any prime $p\notin S\cup \{\ell\}$ and let $\alpha,\beta,\gamma \in \FFbar_\ell$ be the roots of $\det(x I -\rho_1(\Frob_p)) \in \FF_\ell[x]$.  Define the values $u:=\alpha+\beta+\gamma$ and $v:=\alpha\beta+\alpha\gamma+\beta\gamma$; they belong to $\FF_\ell$.  We have $\alpha\beta\gamma=\det(\rho_1(\Frob_p))=\chi_\ell(\Frob_p)^e\equiv p^e \pmod{\ell}$.    

Using Lemma~\ref{L:dim 3} and $\alpha\beta\gamma=p^e$, we find that the polynomial $P_p(T)$ modulo $\ell$ is equal to
\begin{align*}
&(T-\alpha)(T-\beta)(T-\gamma)(T-p/\alpha)(T-p/\beta)(T-p/\gamma)\\
=&(T^3-uT^2+v T - p^e)(T^3-p^{1-e} v T^2+p^{2-e} u T - p^{3-e})\\
=&T^6-(p^{1-e}v+u) T^5 + (p^{2-e} u + p^{1-e}uv+v) T^4 - (p^{3-e} +p^{2-e} u^2+p^{1-e}v^2+p^e) T^3 + \ldots.
\end{align*}

With $p=2$ and using the coefficients of $P_2(T)$ given in \S\ref{S:good}, we find that  for some $e\in \{0,1\}$, there are $u,v\in \FF_\ell$ such that
\begin{align} \label{E:dim 3}
2^{1-e}v+u = -3,\quad 2^{2-e} u + 2^{1-e}uv+v = 6,\quad 2^{3-e} +2^{2-e} u^2+2^{1-e}v^2+2^e=-9.
\end{align}

First consider the case $e=1$.   The equations (\ref{E:dim 3}) become
\begin{align*} 
v+u +3=0,\quad 2 u + uv+v - 6=0,\quad 2 u^2+v^2+15=0.
\end{align*}
Substituting $v=-3-u$ into the last two equations and using $\ell\neq 3$, we obtain $u^2 + 2u +9=0$ and $u^2 + 2u + 8=0$.  Therefore, $1=  (u^2 + 2u +9)-(u^2 + 2u + 8)=0-0=0$ which gives a contradiction.

We thus have $e=0$.   The equations (\ref{E:dim 3}) become
\begin{align*} 
2v+u+3=0,\quad 4u + 2uv+v -6=0,\quad 4 u^2+2v^2 +18=0.
\end{align*}
Substituting $u=-2v-3$ into the last two equations and using $\ell>3$, we obtain $4v^2 + 13v + 18=0$ and $3v^2 + 8v + 9=0$.  Therefore,
\[
 0 = 3(4v^2 + 13v + 18)-4(3v^2 + 8v + 9) = 7v + 18.
\]
Since $\ell \neq 7$, we have $v=-18/7$.   So $0= 3v^2 + 8v + 9 =3^4 5/7^2$ in $\FF_\ell$, which is a contradiction since $\ell>7$.

This shows that the case $d_1=3$ does not occur.

%-----------------------------------------------------------------------------------------------------------------------------------------------------------
\section{Primitivity} \label{S:primitive}

In this section, we prove the following:

\begin{prop} \label{P:primitive}
The action of $G_\QQ$ on $J[\ell]$ is primitive for all odd primes $\ell$.
\end{prop}

Suppose that there is an odd prime $\ell$ for which the action of $G_\QQ$ on $J[\ell]$ is imprimitive.   Hence there is an integer $r\geq 2$ and non-zero $\FF_\ell$-subspaces $W_1,\ldots, W_r$ of $J[\ell]$ such that $J[\ell] = W_1 \oplus \cdots \oplus W_r$ and such that 
\[
\{\sigma(W_1),\ldots, \sigma(W_r)\}= \{W_1,\ldots, W_r\}
\]
for all $\sigma\in G_\QQ$.   The $G_\QQ$-action on the set $\{W_1,\ldots, W_r\}$ must be transitive since $G_\QQ$ acts irreducibly on $J[\ell]$ by Proposition~\ref{P:irreducible}.    In particular, $\dim_{\FF_\ell} W_i$ is independent of $i$  and hence equals $6/r$.   Therefore, $r\in \{2,3,6\}$.

\begin{lemma}   \label{L:prime exclusions 2}
We have $\ell \notin \{3,5,7,11,83\}$.
\end{lemma}
\begin{proof}
Suppose $\ell \in \{3,5,7,11,83\}$.   We claim that there is a prime $p\notin S \cup \{\ell\}$ such that the polynomial $P_p(x) = x^6 + a_p x^5 + b_p x^4 +c_p x^3 + p b_p x^2 +p^2 a_p x+ p^3$ is irreducible in $\FF_\ell[x]$ and such that $\ell \nmid a_p$.   From the polynomials given in \S\ref{S:good}, the claim is true with $p=2$ if $\ell \in \{7,11\}$, $p=17$ if $\ell \in \{3\}$, $p=19$ if $\ell \in \{83\}$ and $p=43$ if $\ell \in \{5\}$.

We have $\tr(\rho_{J,\ell}(\Frob_p))\equiv -a_p \not\equiv 0 \pmod{\ell}$.   The matrix $\rho_{J,\ell}(\Frob_p)$ permutes the spaces $W_1,\ldots, W_r$.   The matrix $\rho_{J,\ell}(\Frob_p)$ thus stabilizes some $W_j$ since otherwise $\tr(\rho_{J,\ell}(\Frob_p))=0$.    However, this is impossible since $\det(xI - \rho_{J,\ell}(\Frob_p))\equiv P_p(x) \pmod{\ell}$ is irreducible.     Therefore, $\ell \notin \{3,5,7,11,83\}$.
\end{proof}

The action of $G_\QQ$ on the set $\{W_1,\ldots, W_r\}$ can be expressed as a representation
\[
\varphi\colon G_\QQ \to \mathfrak{S}_r,
\]
i.e., $\sigma(W_i)= W_{\varphi(\sigma) i}$ for $1\leq i \leq r$ and $\sigma \in G_\QQ$.

\begin{lemma} \label{L:phi unramified}
The representation $\varphi$ is unramified at all primes $p$.
\end{lemma}
\begin{proof}
The representation $\varphi$ factors through $\rho_{J,\ell}$.     Therefore, $\varphi$ is unramified at all primes $p\in S \cup \{\ell\}$.   Suppose that $p \in S$.  Since $\ell \notin S$ by Lemma~\ref{L:prime exclusions 2}, we have $p\neq \ell$ and hence $\rho_{J,\ell}(I_p)$ has order $\ell$ by Proposition~\ref{P:bad primes}.    Therefore, $\varphi(I_p)$ has order $1$ or $\ell$.   We have $r\leq 6$, so $\ell$ does not divide $|\mathfrak{S}_r|=r!$ by Lemma~\ref{L:prime exclusions 2}.  Therefore, $\varphi(I_p)=1$.

Finally suppose that $p=\ell$ and $p\notin S$.    We have $\ell\nmid |\varphi(I_\ell)|$ since $\ell \nmid |\mathfrak{S}_r|$.  Therefore, $\varphi(I_\ell) \subseteq \mathfrak{S}_r$ is cyclic of order at most $r\leq 6$ (as noted in \S\ref{S:inertia at ell}, the tame inertia group at $\ell$ is pro-cyclic).   Let $H$ be the kernel of $\varphi|_{I_\ell}$; we have $[I_\ell : H] = |\varphi(I_\ell)| \leq 6 < \ell-1$.   

Take any $i \in \{1,\ldots, r\}$.   The group $H$ acts on $W_i$ so there is an irreducible $H$-submodule $\calW_i$ of $W_i$.  Define $\calV_i := \sum_{\sigma \in I_\ell} \sigma (\calW_i)$; it is an irreducible $I_\ell$-module.  Lemma~\ref{L:amplitude cor}(\ref{L:amplitude cor ii}) and $[I_\ell:H]<\ell-1$ implies that $\calV_i = \calW_i$.     For any $\sigma\in I_\ell$, we have $\sigma(\calW_i) = \calW_i \subseteq W_i$.  Since $I_\ell$ permutes the spaces $W_1,\ldots, W_r$, we deduce that $\sigma(W_i)= W_i$ for all $\sigma\in I_\ell$.    Since $i$ was arbitrary, we find that $I_\ell$ acts on all the spaces $W_i$ and hence $\varphi(I_\ell)=1$.
\end{proof}

Since $\QQ$ has no non-trivial extensions unramified at all primes, Lemma~\ref{L:phi unramified} implies that $\varphi=1$.   Therefore, $\sigma(W_i) = W_i$ for all $\sigma\in G_\QQ$ and $1\leq i \leq r$.    However, this implies that the action of $G_\QQ$ on $J[\ell]$ is reducible which contradicts Proposition~\ref{P:irreducible}.     Therefore, the action of $G_\QQ$ on $J[\ell]$ is in fact primitive and this completes the proof of Proposition~\ref{P:primitive}.

%-----------------------------------------------------------------------------------------------------------------------------------------------------------
\section{Proof of Theorem~\ref{T:main}}

Take any odd prime $\ell$.   The group $\rho_{J,\ell}(G_\QQ) \subseteq \GSp_{6}(\FF_\ell)$ contains a transvection by Proposition~\ref{P:bad primes}.   By Propositions~\ref{P:irreducible} and \ref{P:primitive}, the representation $\rho_{J,\ell}$ is irreducible and primitive.    By Proposition~\ref{P:mod ell group theory}, we deduce that $\rho_{J,\ell}(G_\QQ) \supseteq \Sp_{6}(\FF_\ell)$.  We also have $\rho_{J,2}(G_\QQ)=\GSp_6(\FF_2)$ by Lemma~\ref{L:mod 2 case}.

From Proposition~\ref{P:reduce to mod ell cases}, we can now conclude that $\rho_J(G_\QQ)=\GSp_6(\Zhat)$.

%-----------------------------------------------------------------------------------------------------------------------------------------------------------

%\bibliographystyle{plain}
%\bibliography{//Users/zywina/Documents/papers/bib/master}

% \bib, bibdiv, biblist are defined by the amsrefs package.

\end{document}